\documentclass[12pt]{amsart}
\usepackage{hyperref}
\usepackage{amsmath}
\usepackage{amsthm}
\usepackage{mathrsfs}
\usepackage{amssymb}
\usepackage{tikz-cd}
\usepackage{comment}
\usepackage{mathtools}
\usepackage{manfnt}
\calclayout
\newcommand{\PP}{\mathbb{P}}
\newcommand{\ZZ}{\mathbb Z}

\newcommand{\codim}{\mathrm{codim\,}}

\newcommand{\ra}{\rightarrow}

\newtheorem{theorem}{Theorem}[section]

\newtheorem{lemma}[theorem]{Lemma}

\newtheorem{corollary}[theorem]{Corollary}

\theoremstyle{definition}

\newtheorem*{remark}{Remark}

\title{Exceptional loci in Lefschetz theory}
\author{Sam Raskin}
\address{Department of Mathematics, The University of Texas at Austin,  RLM 8.100, 2515 Speedway Stop
C1200, Austin, TX 78712}
\email{sraskin@math.utexas.edu}
\author{Geoffrey Smith}
\address{Department of Mathematics, Stat. and CS \\University of Illinois at Chicago, Chicago, IL 60607}
\email{geoff@uic.edu}
\date{}
\begin{document}
\begin{abstract}
Let $\phi:X\ra \PP^n$ be a morphism of varieties. Given a hyperplane $H$ in $\PP^n$, there is a Gysin map from the compactly supported cohomology of $\phi^{-1}(H)$ to that of $X$. We give conditions on the degree of the cohomology under which this map is an isomorphism for all but a low-dimensional set of hyperplanes, generalizing results due to Skorobogatov, Benoist, and Poonen-Slavov. Our argument is based on Beilinson's theory of singular supports for \'etale sheaves. 
\end{abstract}
\maketitle
\section{Introduction and statement of results}
In this note, we prove a generalization of the following theorem due to Benoist and written here in a form due to \cite{PS20}.
\begin{theorem}[\cite{Ben11}, Th\'eor\`eme 1.4]\label{BPS}
Let $X\subset \PP^N$ be a geometrically irreducible quasiprojective variety over a field $k$. Define $M_{\mathrm{bad}}\subseteq \check\PP^N$ as the locus of hyperplanes $H$ such that $X_H:=X\cap H$ is not geometrically irreducible. Then $\codim M_\mathrm{bad}\geq \dim X -1$
\end{theorem}
This result is geometric, but has a cohomological reformulation: the top degree compactly supported cohomology groups of $X $ and $X_H$ are isomorphic. In this note, we prove a similar result for all cohomology of sufficiently high degree on $X$. 

\begin{theorem}\label{main}
Let $X$ be a separated scheme of finite type over a separably closed field $k$, let $\phi:X\ra \PP^n$ be a morphism, and let $\Lambda=\ZZ/\ell\ZZ$ for $\ell$ a prime power not divisible by the characteristic of $k$. Set $r=\dim X\times_{\PP^n} X$. 

Then for each $c\geq 1$, there is a closed subscheme $Z_c\subset \check\PP^n$ of dimension at most $n-c$ such that for $H\in \check\PP^n\setminus Z_c$ and 
$q>c+r$ (resp. $q = c+r$), the Gysin map
\[
H^{q-2}_c(\phi^{-1}(H), \Lambda(-1))\ra H^q_c(X, \Lambda)
\]
is an isomorphism (resp. surjective). 
\end{theorem}

In the setting of Theorem \ref{BPS}, taking $q = 2\dim X$, 
we see that $M_\mathrm{bad} \subset Z_{\dim X-1}$, 
which has dimension $\leq n -\dim X+1$,  
recovering the assertion of \emph{loc. cit}.

We remind the construction of Gysin maps in Section 2.

\begin{remark}
If $X$, $\phi$, and $H$ are defined over an arbitrary field $k'$ with separable
closure $k$, then our construction shows that $Z_c$ is naturally defined over $k'$.
Moreover, a Gysin map $H^{q-2}_c(\phi^{-1}(H)\times_{\rm{Spec}(k')} \mathrm{Spec}(k), \Lambda(-1))\ra H^q_c(X\times_{\rm{Spec}(k')} \rm{Spec}(k), \Lambda)$ is compatible with the action of $\rm{Gal}(k/k')$, so if $H \not \in Z_c$ is defined over $k'$ and $q$ is as 
in Theorem \ref{main}, the Gysin map is an isomorphism or surjection of Galois representations.
\end{remark}
The main new tool in the proof of this result is Beilinson and Saito's works \cite{Bei16, Sai17} on the singular support of constructible sheaves in arbitrary characteristic. In Section 2, we will use their work to prove a result, Theorem \ref{finalPiece}, that is at the core of the argument; we will also collect a couple lemmas we will need. In Section 3, we will use these tools to prove Theorem \ref{main}.
\subsection{Past results}
 Theorem \ref{main} generalizes a number of results beyond Theorem \ref{BPS}. Poonen-Slavov \cite{PS20} establish the  $q=2\dim X$ case of Theorem \ref{main} under the additional assumption that $\phi$ has equidimensional fibers. If $\phi$ is the immersion of a normal projective complex variety and $c=1$, Theorem \ref{main} is Corollary 7.4.1 of \cite{GM83}, and is proven as a special case of the paper's Lefschetz hyperplane theorem for intersection homology.
 
 And if $\phi$ is the closed immersion of a smooth projective variety, this result is known; by \cite[Theorem 2.1]{Sko92}, there exists an isomorphism $H^{q-2}_c(\phi^{-1}(H), \Lambda(-1))\ra H^q_c(X, \Lambda)$ so long as $q\geq n+s+3$, where $s$ is the dimension of the singular locus of $\phi^{-1}(H)$. So the locus where there exists no isomorphism $H^{q-2}_c(\phi^{-1}(H), \Lambda(-1))\ra H^q_c(X, \Lambda)$ is contained in the locus of hyperplanes such that $\phi^{-1}(H)$ is singular in dimension at least $q-n-2$. This locus in turn has codimension at least $q-n-1$.

\subsection*{Acknowledgements}
 We would like to thank Sasha Beilinson, Robert Cass, Joe Harris, Bjorn Poonen, and Alexander Smith for helpful conversations.
 \section{The Gysin map for bounded complexes}
\subsection{Notation and the basic setup}
Let $k$ be a separably closed field, let $\ell$ be a prime power not divisible by the characteristic of $k$, and set $\Lambda=\ZZ/\ell \ZZ$. For the remainder of this paper, sheaves will be constructible \'etale sheaves of $\Lambda$-modules, and for any variety $V$ we will use $D(V)$ to denote the bounded derived category of constructible sheaves of $\Lambda$-modules on $V$. Given a $k$-point $H\in \check \PP^n$, let $i:H\ra \PP^n$ denote the corresponding inclusion, and let $j:\PP^n\setminus H\ra \PP^n$ be the inclusion of the complement. For the remainder of the paper, all functors will be derived; for instance, we will use $j_*$ to denote the derived pushforward associated to $j$. Finally, we will use the notation $\mathbb{H}^q(\mathcal{F})$ to denote the hypercohomology of $\mathcal{F}\in D(V)$ in degree $q$. 

On the level of sheaves, the map in Theorem \ref{main} is induced by applying $R\Gamma$ to a composition of two arrows in $D(\PP^n)$. Given a bounded complex of sheaves $\mathcal{F}$ on $\PP^n$, the two maps we will use are:
\begin{itemize}
    \item The counit map $i_!i^!\mathcal{F}\ra \mathcal{F}$, which appears in the localization triangle
    \[
    i_!i^!\mathcal{F}\ra \mathcal{F}\ra j_*j^*\mathcal{F}\overset{+1}{\ra}.
    \]
    \item The Gysin map $i^*\mathcal{F}(-1)[-2]\ra i^! \mathcal{F}$, as discussed in \cite{Fuj02}.
\end{itemize}

Applying $R\Gamma$ to these two maps produces maps $$\mathbb{H}^{q}(H,i^!\mathcal{F})\ra \mathbb{H}^q(\mathcal{F})$$ and $$\mathbb{H}^{q-2}(H, i^*\mathcal{F}(-1))\ra \mathbb{H}^q(H,i^!\mathcal{F})$$ respectively. Our main effort will be proving that these maps are isomorphisms or surjections assuming certain hypotheses on $\mathcal{F}$ and $H$.

\subsection{The counit map}
Of the two maps above, the counit map is the easier one to understand, as its cone is $j_*j^*\mathcal{F}$.
\begin{lemma}\label{counit}
Let $\mathcal{F}\in D(\PP^n)$. Fix an integer $r$. Suppose that for all $p$, the sheaf $\mathcal{H}^p(\mathcal{F})$ has support of dimension at most  $r-p$. Then the sheaf $\mathcal{H}^p(j_*j^*\mathcal{F})$ has support of dimension at most $r-p$ and we have $\mathbb{H}^q(\PP^n,j_*j^*\mathcal{F})=0$ for $q>r$.
\end{lemma}
\begin{proof}
Let $\mathcal{F}$ and $r$ satisfy the hypotheses of the lemma. For each $p$, the support of $\mathcal{H}^p(j^*\mathcal{F})\cong j^*\mathcal{H}^p(\mathcal{F})$ has dimension at most $r-p$. Since $\PP^n\setminus H$ is affine, we then have that affine vanishing \cite[Theorem VI.7.3]{Mil80} implies that $R^sj_*(\mathcal{H}^p(j^*\mathcal{F}))$ is supported in dimension at most $r-p-s$ and $H^s(\mathcal{H}^p(j^*\mathcal{F}))=0$ if $s+p>r$. The results follow by applying a spectral sequence to the filtered complex $\tau^{\geq p} \mathcal{F}$.
\end{proof}

\begin{remark}
The hypothesis of the lemma is equivalent to asking that $\mathcal{F}$ sit
in perverse degrees $\leq r$. From this perspective, the claim follows from
right exactness of affine pushforwards with respect to the perverse $t$-structure.
\end{remark}

\subsection{The Gysin map}
The following result is the most nonstandard ingredient of the proof of Theorem \ref{main}.
\begin{theorem}\label{finalPiece}
Let $\mathcal{F}$ be an object of $D(\PP^n)$, let $\Phi\subset \PP^n\times \check \PP^n$ be the universal hyperplane, and let $\pi_1$, $\pi_2$ denote the projections. Then there is a 
closed subscheme $Z\subset \Phi$ of dimension $\leq n-1$ 
such that for $H\in \check \PP^n$, the cone of the Gysin map
\[
i^*\mathcal{F}(-1)[-2]\ra i^! \mathcal{F},
\]
associated to $i:H\ra \PP^n$ is supported on $\pi_1(Z\cap \pi_2^{-1}(H))$. 
\end{theorem}
In the proof of Theorem \ref{main}, we will use the following corollary of this theorem.
\begin{corollary}\label{finalPieceCorollary}
With $\mathcal{F}$ as in Theorem \ref{finalPiece}, for any positive integer $c$ there is a closed subscheme $Z_c\subset \check \PP^n$ of dimension at most $n-c$ such that for any $H\in \check \PP^n\setminus Z_c$ the cone of the Gysin map associated to $i:H\ra \PP^n$ is supported on a subscheme of dimension at most $c-2$.
\end{corollary}

\begin{proof}
Take $Z$ as in Theorem \ref{finalPiece}.  
Let $Z_c \subset \check{\PP}^n$ be the closed subscheme where the fibers
of the projection $Z \to \check{\PP}^n$ have dimension $\geq c-1$. 
As $\dim Z = n-1$, we have $\dim Z_c \leq n-c$, and it satisfies the conclusion
by definition of $Z$. 
\end{proof}

The theory of \emph{singular support} for \'etale sheaves was developed by Beilinson \cite{Bei16} and Saito \cite{Sai17}. We use it in the following form.
\begin{theorem}[Theorem 1.3, \cite{Bei16}]\label{beilinson}
Let $X$ be a smooth variety of dimension $n$ and let
$\mathcal{F}$ be a bounded constructible complex on it. 

Then there exists a closed, conical subscheme $SS(\mathcal{F}) \subset T^*X$ 
of dimension $n$ such that for every pair
$h:U\ra X$ smooth and $f:U\ra Y$ a morphism to a smooth variety $Y$ such that $df(T^*Y)\cap dh(SS(\mathcal{F}))$ is contained in the zero section 
$U \subset T^*U$, the map $f$ is locally acyclic with respect to $h^*(\mathcal{F})$ (in the sense of \cite{Del77}, 2.12). 
\end{theorem}

\begin{remark}
The most serious part of the theorem is the calculation of the dimension 
of $SS(\mathcal{F})$.
\end{remark}

\begin{lemma}\label{l:gysin-ss}

Suppose $X$ is smooth and let $i:D \to X$ be the embedding of a smooth divisor. 
Suppose $\mathcal{F} \in D(X)$ has the property that the intersection 
$$SS(\mathcal{F})|_D \cap N_{X/D}^* \subset T^*X \underset{X}{\times} D$$
is contained in the zero section. 

Then the Gysin map
$$i^*\mathcal{F}(-1)[-2]\ra i^! \mathcal{F}$$
is an isomorphism.

\end{lemma}

\begin{proof}

This result is a special case of \cite{Sai17} Proposition 7.13. We outline
a direct proof below. 

The assertion is local on $X$, so we may assume there is a map
$f:X \to \mathbb{A}^1$ with $D = f^{-1}(0)$. 

By assumption, the image of $\mathbb{P} (SS(\mathcal{F}) \cap df(X)) \to X$
avoids $D$, so we may assume this intersection is zero.
Then taking
$U = X$ and $Y = \mathbb{A}^1$ in the definition of
singular support, we see that $\mathcal{F}$ is locally acyclic
with respect to $f$ by hypothesis, so its vanishing cycles 
$\phi_f(\mathcal{F})$ are $0$.

In general, recall that we have exact triangles
$$i^*(\mathcal{F}) \to \psi_f(\mathcal{F}) \overset{\operatorname{can}}{\to} 
\phi_f(\mathcal{F}) \overset{+1}{\to} $$
$$\phi_f(\mathcal{F}) \overset{\operatorname{var}}{\to} 
\psi_f(\mathcal{F})(-1) \to i^!(\mathcal{F})[2]  \overset{+1}{\to}$$
such that the composition 
$$i^*(\mathcal{F})(-1)[-2] \to \psi_f(\mathcal{F})(-1)[-2] \to i^!(\mathcal{F})$$
is the Gysin map. As $\phi_f(\mathcal{F}) = 0$ by the above, we obtain our
claim.

\end{proof}

\begin{proof}[Proof of Theorem \ref{finalPiece}]
Recall that $\Phi$ is canonically isomorphic to $\PP T^*\PP^n$. Let
$Z=\PP SS(\mathcal{F})$ as in Theorem \ref{beilinson}. By the theorem,
it has dimension $n-1$.

Fix $H\in \check \PP^n$, and let $B\subset H$ be the subscheme 
$\pi_1(Z\cap \pi_2^{-1}(H))$. By construction, 
the closed embedding $H \setminus B \to \PP^n \setminus B$ satisfies 
the hypotheses of Lemma \ref{l:gysin-ss}, so the cone of the Gysin map
is supported on $B$ as desired.

\end{proof}

\section{Proof of Theorem \ref{main}}
In this section, we continue to use notation from the previous section.

Let $\phi:X\ra \PP^n$ be a morphism of separated schemes of finite type over $k$. In this section we establish Theorem \ref{main} for the map $\phi$. Let $\Lambda_X$ denote the constant sheaf on $X$ with fiber $\Lambda$.

We first note that $R^p\phi_!\Lambda_X$ is supported at points over which $\phi$ has fiber dimension at least $\frac{p}{2}$, because the stalk of $R^p\phi_!\Lambda_X$ at a point $x$ is just $H^p_c(\phi^{-1}(X),\Lambda)$ by \cite[Lemma 0F7L]{Stacks}.
As a consequence, if $X\times_{\PP^n} X$ has dimension $r$, we have that $\mathcal{H}^p(\phi_!\Lambda_{X})$ has support in dimension at most $r-p$. So by Lemma \ref{counit} we have that for any inclusion of a hyperplane $i:H\ra \PP^n$, the hypercohomology groups $\mathbb{H}^q(\PP^n,j_*j^*\phi_!\Lambda_X)$ vanish for $q>r$. So we have that the counit map $\mathbb{H}^q(H, i^! \phi_!\Lambda_{X})\ra H_c^q(X, \Lambda)$ is an isomorphism if $q\geq r+2$ and is a surjection if $q=r+1$. 

We now show that the Gysin map induces an isomorphism or surjection on cohomology. Fixing $c\geq 1$, let $Z_c$ be the exceptional set in Corollary \ref{finalPieceCorollary} relative to the complex of sheaves $\phi_! \Lambda_X$. Fix some $k$-point $H\in \check \PP^n\setminus Z_c$, and let $Q\in D(H)$ denote the cone of the morphism
\[
i^*\phi_!\Lambda_{X}(-1)[-2]\ra i^!\phi_!\Lambda_{X}.
\]
By Corollary \ref{finalPieceCorollary}, $Q$ is supported on a closed subscheme $B$ of dimension at most $c-2$. Moreover, $\mathcal{H}^p(j_*j^*\phi_!\Lambda_{X})$ has support of dimension at most $\min(r-p,n)$ by Lemma \ref{counit}. Since we also have that $R^p\phi_!\Lambda_{X}$ is supported in dimension $\min(r-p, n)$, the distinguished triangle $i_!i^! \phi_!\Lambda_{X}\ra \phi_! \Lambda_X\ra j_*j^*\phi_! \Lambda_{X}\ra $ gives the bound on the dimension of the support of $\mathcal{H}^p(i_!i^! \phi_!\Lambda_{X})$,
\[
\mathrm{dim\ supp}(\mathcal{H}^p(i_!i^! \phi_!\Lambda_{X}))\leq \min(r-p+1,n).
\]

Likewise, $\mathcal{H}^p(i^*\phi_*\Lambda_{X}(-1)[-2])$ is supported on a set of dimension at most $\min(r-p+2,n-1)$. From these two observations and the defining triangle for $Q$, we see that $\mathcal{H}^p(Q)$ is supported on a subscheme of dimension at most $\min(r-p+1, c-2)$. 
Therefore, $\mathbb{H}^{q-p}(H,\mathcal{H}^p(Q)) = 0$ for $q-p > 2\min(r-p+1,c-2)$.
Observe that if $q-p \leq 2\min(r-p+1,c-2)$, then $2(q-p) \leq 2(r-p+1)+2(c-2)$, 
so $q \leq r+c-1$. Therefore, whenever $q>r+c-1$, $\mathbb{H}^{q-p}(H,\mathcal{H}^p(Q)) = 0$.
Applying the Grothendieck spectral sequence, we find that $\mathbb{H}^q(H,Q) = 0$ for $q>r+c-1$. 

From this, we conclude that the Gysin map
\[
H^{q-2}_c(\phi^{-1}(H),\Lambda(-1))\cong \mathbb{H}^{q-2}(H, i^*\phi_!\Lambda_{X}(-1))\ra \mathbb{H}^q(H,i^!\phi_!\Lambda_{X})
\]
is an isomorphism for $q>r+c$ and a surjection if $q=r+c$. Combining this with the counit map above, and noting $c\geq 1$, we have that the map on cohomology
\[
H^{q-2}_c(\phi^{-1}(H),\Lambda(-1))\ra H_c^q(X, \Lambda)
\]
is an isomorphism for $q>r+c$ and a surjection if $q=r+c$, proving Theorem \ref{main} for $X$.

\begin{remark}

In the above setting, note that $\phi_!(\Lambda_X) \in D(\PP^n)$
is in perverse degrees $\leq r$. 
 The argument shows that, suitably understood, the 
conclusion of Theorem \ref{main} holds for
any $\mathcal{F} \in D(\PP^n)$ a bounded complex of constructible sheaves
in perverse degrees $\leq r$.

\end{remark}
\bibliographystyle{plain}
\bibliography{slice}

\begin{thebibliography}{10}

\bibitem{Bei16}
Alexander Beilinson.
\newblock Constructible sheaves are holonomic.
\newblock {\em Selecta Mathematica}, 22(4):1797--1819, 2016.

\bibitem{Ben11}
Olivier Benoist.
\newblock Le th{\'e}or{\`e}me de {B}ertini en famille.
\newblock {\em Bulletin de la Soci{\'e}t{\'e} Math{\'e}matique de France},
  139(4):555--569, 2011.

\bibitem{Del77}
Pierre Deligne.
\newblock Th\'eor\`emes de finitude en cohomologie $\ell$-adique.
\newblock In {\em Cohomologie \'etale: S\'eminaire de G\'eom\'etrie
  Alg\'ebrique du {B}ois-{M}arie {SGA} 4 1/2}, pages 233--251. Springer, 1977.

\bibitem{Fuj02}
Kazuhiro Fujiwara.
\newblock A proof of the absolute purity conjecture (after {G}abber).
\newblock In {\em Algebraic Geometry 2000, Azumino}, pages 153--183, Tokyo,
  Japan, 2002. Mathematical Society of Japan.

\bibitem{GM83}
Mark Goresky and R~MacPherson.
\newblock Intersection homology {II}.
\newblock {\em Inventiones Mathematicae}, 72:77, 1983.

\bibitem{Mil80}
James~S. Milne.
\newblock {\em \'Etale Cohomology}.
\newblock Princeton University Press, 1980.

\bibitem{PS20}
Bjorn Poonen and Kaloyan Slavov.
\newblock The exceptional locus in the {B}ertini irreducibility theorem for a
  morphism.
\newblock {\em arXiv preprint arXiv:2001.08672}, 2020.

\bibitem{Sai17}
Takeshi Saito.
\newblock The characteristic cycle and the singular support of a constructible
  sheaf.
\newblock {\em Inventiones Mathematicae}, 207:597--695, 2017.

\bibitem{Sko92}
Alexei~N Skorobogatov.
\newblock Exponential sums, the geometry of hyperplane sections, and some
  {D}iophantine problems.
\newblock {\em Israel Journal of Mathematics}, 80(3):359--379, 1992.

\bibitem{Stacks}
The {Stacks Project Authors}.
\newblock \textit{Stacks Project}.
\newblock \url{https://stacks.math.columbia.edu}, 2020.

\end{thebibliography}
\end{document}